\documentclass[12pt,reqno]{amsart}
\usepackage{amsopn,amssymb,mathrsfs,mathrsfs,a4wide,color,url,microtype}
\usepackage[unicode,linktocpage]{hyperref}

\newcommand{\bib}{\bibitem}

\def\card{\operatorname{card}}

\def\vep{\varepsilon}

\def\sup{\operatorname{sup}}
\def\supp{\operatorname{supp}}

\def\en{\mathbb{N}}
\def\er{\mathbb{R}}
\def\om{\omega}
\def\mc{\mathcal}

\def\Int{\operatorname{Int}}

\def\vvs{\vskip 2 mm}

\newtheorem{theorem}{Theorem}
\newtheorem{lemma}[theorem]{Lemma}
\newtheorem{proposition}[theorem]{Proposition}
\newtheorem{corollary}[theorem]{Corollary}
\theoremstyle{remark}
\newtheorem*{remark*}{Remark}

\newcommand\absb[2]{\csname#1l\endcsname|#2\csname#1r\endcsname|}
\newcommand\norm[1]{\mathopen\|#1\mathclose\|}

\newcommand\normb[2]{\csname#1l\endcsname\|#2\csname#1r\endcsname\|}

\newcommand\ga{\gamma}

\newcommand\R{\mathbb R}

\allowdisplaybreaks

\begin{document}

\title{On uniformly differentiable mappings from $\ell_\infty(\Gamma)$}

\begin{abstract}
In 1970 Haskell Rosenthal proved that if  $X$ is a Banach space,
$\Gamma$ is an infinite index set, and $T:\ell_\infty(\Gamma)\to X$
is a bounded linear operator such that
$\inf_{\gamma\in\Gamma}\|T(e_\gamma)\|>0$
then $T$ acts as an isomorphism on $\ell_\infty(\Gamma')$, for some
$\Gamma'\subset\Gamma$ of the same cardinality as $\Gamma$. Our main result
is a nonlinear strengthening of this theorem. More precisely,
under the assumption of GCH and the regularity of $\Gamma$, we show
that if  ${F}:B_{\ell_\infty(\Gamma)}\to X$ is uniformly differentiable and
 such that $\inf_{\gamma\in\Gamma}\|{F}(e_\gamma){-F(0)}\|>0$ then
 there exists $x\in B_{\ell_\infty(\Gamma)}$ such that $d{F}(x)[\cdot]$ is
 a bounded linear operator which
 acts as an isomorphism on $\ell_\infty(\Gamma')$, for some
$\Gamma'\subset\Gamma$ of the same cardinality as $\Gamma$.
\end{abstract}

\author{Petr H\'ajek}
\address{Mathematical Institute, Czech Academy of Science, \v{Z}itn\'{a} 25, 115 67 Praha 1,
Czech Republic, and Department of Mathematics, Faculty of Electrical Engineering, Czech Technical University in Prague, Zikova 4, 160 00, Prague}
\email{hajek@math.cas.cz}

\author{Eva Perneck\'a}

\address{ Institute of Mathematics, Polish Academy of Sciences, ul. \'{S}niadeckich 8, 00-956 Warszawa, Poland.}
\email{e.pernecka@impan.pl}

\thanks{The work was supported in part by GA\v CR 16-073785, RVO: 67985840 and SVV-2014-260106.}
\maketitle

\section{Introduction}
The structural stability of $C(K)$ spaces under bounded
linear operators has been studied extensively in the past.
According to Pe\l czy\'nski, if $X$ is a Banach space and
$T: c_0\to X$ is a non-compact linear operator,
then $c_0$ contains a linear subspace $Z$ isomorphic to $c_0$ such that
$T\upharpoonright_Z$ is an isomorphism (see \cite{Pel}, \cite[Theorem 4.51]{FHHMZ}).
In particular, $X$ contains a copy of $c_0$.
Similarly, if  $T:\ell_\infty\to X$ is a
non-weakly compact linear operator, then $\ell_\infty$
contains a linear subspace $Z$ isomorphic to $\ell_\infty$ such that
$T\upharpoonright_Z$ is an isomorphism (see \cite[Proposition 2.f.4]{LT}).
In particular, $X$ again contains a copy of $\ell_\infty$. Another version
of this result was shown by Rosenthal. We denote by
$e_\gamma=\chi_{\{\gamma\}}\in\ell_\infty(\Gamma), \gamma\in\Gamma$,
the standard unit vectors.

\begin{theorem}\label{rose}\cite{R}
If $\Gamma$ is an infinite set and
$T:\ell_\infty(\Gamma)\to X$ is a bounded linear operator
such that  $\inf_{\gamma\in\Gamma}\|T(e_\gamma)\|>0$ then there exists
$\Gamma'\subset\Gamma$ of the same cardinality as $\Gamma$ such that
$T\upharpoonright_{\ell_\infty(\Gamma')}$ is an isomorphism.
\end{theorem}

To be precise, in the statement of \cite[Proposition 1.2]{R}, the assumption
is that  $T\upharpoonright_{c_0(\Gamma)}$ is an isomorphism. However, it is
noted later on that the weaker assumption (as stated here) is sufficient.

The  Pe\l czy\'nski-type result has been
generalized into the setting of uniformly differentiable mappings in
\cite{DH}, \cite{CHL}, \cite[Theorem 6.45]{HJ}. The
Rosenthal-type result will be generalized in the present note.

We refer to \cite{HJ} for  background in smoothness and to \cite{FHHMZ}
for  background in general Banach space theory. Facts
concerning set theoretical aspects can be found in \cite{J}.

For the purposes of this note
we will recall some of the basic concepts.

A non-decreasing function
$\omega\colon[0,+\infty)\to [0,+\infty)$ continuous at $0$
with $\omega(0)=0$ is called a \emph{modulus}.
Let $(P,\rho)$ and $(Q,\sigma)$ be metric spaces.
We say that a
continuous mapping $f\colon P\to Q$ is uniformly continuous if
there exists a modulus function $\omega_f$ such that
$$
\omega_f(\delta)\ge\sup\{\sigma(f(x),f(y)), x,y\in P, \rho(x,y)\leq\delta\}\;{\text{for all}\;\delta\geq 0}.
$$
If the above inequality holds we will say that $f$ has modulus
of continuity $\omega_f$. In this note we are not concerned with
finding the best modulus which can be applied to a given mapping $f$.

By $X,Y$ we will always denote real Banach spaces,
$B_X$ stands for the closed unit ball.

Let $V\subset X$ be convex with \mbox{non-empty} interior $\Int V$.
 For a given continuous mapping $f:V\to Y$ we denote by
 $df(x)\in\mathcal{L}(X;Y)$ the bounded
 linear operator which is the Fr\'echet
 derivative at $x\in\Int V$ (if it exists). By $df(x)[h]$ we denote
 the evaluation of the derivative in the direction $h\in X$.
By $\mathcal C^{1,+}(V;Y)$ we denote the space of
 continuous mappings from $V$ into $Y$ which have a uniformly continuous
 derivative in the interior $\Int V$. In the scalar case, we use shortened notation
 $\mathcal C^{1,+}(V)=\mathcal C^{1,+}(V;\er)$. It is well-known and easy to
 check that uniformly differentiable mappings from bounded sets are Lipschitz.

Let us point out, without
going into details, that Theorem 6.45 in \cite{HJ} implies that if
$f\in\mathcal C^{1,+}(B_{c_0};Y)$ is a non-compact mapping, then there
exists a point $x^{**}\in B_{c_0^{**}}$ such that
$d\left(f^{**}\right)(x^{**})$ is a non-weakly compact bounded linear
operator from $\ell_\infty$ into $Y^{**}$.
(For the precise meaning of the bidual mapping $f^{**}$ see \cite{HJ}
Chapter 6.) In particular, by the result mentioned above
on the linear operators from $\ell_\infty$, $d f^{**}$ fixes a copy
of $\ell_\infty$, hence $Y^{**}$
contains a copy of $\ell_\infty$.
It remains an open question if $Y$ itself contains a copy of $c_0$.
It should be noted that the problem cannot be solved by
means of differentiation, in view of the next simple example. Indeed,
choosing
a surjective increasing $C^\infty$-smooth function $\phi\colon\er\to\er$
such that $\phi(0)=0$ and $d\phi(0)=0$,
one can show that the mapping $\Phi\colon c_0\to c_0$ defined by
$\Phi ((x_k)_{k=1}^\infty)=(\phi(x_k))_{k=1}^\infty$
belongs to $\mc C^{1,+}(c_0;c_0)$, it is surjective, but $d\Phi(x)$
is a compact linear operator from $c_0$ into $c_0$ for every $x\in c_0$.

Recall that $\ell_2$ is a linear quotient of $\ell_\infty$
\cite[p. 141]{HMVZ}.
Therefore, by \cite{H2} (or \cite[Theorem 6.68]{HJ}), there exists a surjective
second degree polynomial from $\ell_\infty$ onto $\ell_1$, and hence a surjective
second degree polynomial onto any separable Banach space.
So the non-weak compactness of
{$\overline{F(B_{\ell_\infty})}$}
is not sufficient for concluding that $Y$ contains a copy of $\ell_\infty$.

In the present note
we will give a uniformly differentiable
variant of Rosenthal's theorem when the initial space is
$\ell_\infty(\Gamma)$. For convenience, as it is usual in this area,
we will assume that the index set
$\Gamma$ is the minimal ordinal representing the cardinal $\card{\Gamma}$.
In order to avoid confusion we will use the symbol $\omega_0$ for the smallest
infinite ordinal, while the symbol $\omega$ will denote modulus function.

In addition, by passing to
infinite-dimensional case via {ultraproducts}, in the last part of our note
we will derive a finite-dimensional counterpart of the result.
To this end we will generalise \cite[Theorem 6.45]{HJ} for uniformly
differentiable mappings
which are not necessarily bidual extensions of uniformly differentiable
mappings.

\section{Auxiliary results}

Let us formulate some results on a long (i.e. transfinite)
 Schauder basic sequence
$\{x_\gamma\}_{\gamma<\Gamma}$
in a Banach space $X$ which will be useful for our purposes.
For the definition and more details on the long Schauder basis
see \cite[p. 132]{HMVZ}.
In particular,  if the long Schauder
basic sequence  is
seminormalized,
i.e. $0<\eta<\|x_\gamma\|<\rho<\infty$, for all $\gamma\in\Gamma$, 
then {there} is
a seminormalized biorthogonal long sequence $\{\phi_\gamma\}_{\gamma<\Gamma}$
of functionals in $X^*$ \cite[p. 134]{HMVZ}.

 We say that
$\{x_\ga\}_{\ga<\Gamma}\subset X$ is
a weakly null (long) sequence provided for every $\phi\in X^*$,
$\lim_{\ga\to\Gamma}\phi(x_\ga)=0$.

\begin{proposition}\label{w-null}
Let $X$ be a Banach space, $\Gamma$ an infinite regular cardinal and
let $\{x_\ga\}_{\ga<\Gamma}\subset X$ be a seminormalized
 weakly null sequence.
Then there is a long Schauder basic subsequence
$\{y_\ga\}_{\ga<\Gamma}$ of
$\{x_\ga\}_{\ga<\Gamma}$. In particular, there is a bounded set
$\{\phi_\ga\}_{\ga<\Gamma}\subset X^*$ which is biorthogonal to
 $\{y_\ga\}_{\ga<\Gamma}$,
i.e. $\phi_\gamma(y_\alpha)=\delta_{\gamma,\alpha}$.
\end{proposition}
\begin{proof}
If $\Gamma=\omega_0$, the result is classical \cite[p. 195]{FHHMZ}.
If $\Gamma>\omega_0$ then the result was proved by Talponen 
\cite[Theorem 1.1]{T}.
\end{proof}

The key ingredient of the proof of the main result, Theorem
\ref{thm:infinite}, will be the next lemma
originating from \cite{Ha}. Our formulation below is formally
somewhat stronger (due to the use of $\eta$), but in fact the
proof of the original result yields our conclusion as well.
Alternatively, it is rather easy to derive the stronger version
by bootstrapping of the original lemma, in a similar way as in the proof
of Lemma \ref{main2}.
We will denote by $\{e_j\}_{j=1}^{n}$ the unit basis of $\ell_\infty^n$.

\begin{lemma}\label{main} {\cite[Lemma 6.27]{HJ}}

For each modulus $\omega$,
for every $L>0$ and every $\vep>0$, there is an $M(\omega,L,\vep)\in\en$
such that if $n\geq M(\omega,L,\vep)$ and if
$f\in\mathcal C^{1,+}(B_{\ell_\infty^n})$ is an $L-$Lipschitz
function whose derivative $df$ has modulus of continuity $\omega$,
then there exists $j\in\{1,\dots,n\}$ for which
\[
|f(\eta e_j)-f(0)|<\vep,\;{\text{whenever}}\;|\eta|\le1.
\]
\end{lemma}

\begin{lemma}\label{main2}

{For each modulus $\omega$,
for every $L>0$,  and every $\vep>0$, there is an
$N(\omega,L,\vep)\in\en$
such that if $n\geq N(\omega,L,\vep)$ and if
$f\in\mathcal C^{1,+}(B_{\ell_\infty(\Gamma)})$, where $\Gamma$ is an infinite regular cardinal, is an $L-$Lipschitz
function whose derivative $df$ has  modulus of continuity $\omega$,
then given any $\{v_k\}_{k=1}^n\subset B_{\ell_\infty(\Gamma)}$ such that $\supp(v_k)\cap\supp(v_l)=\emptyset, k\neq l$, there exists $k\in\{3,\dots,n\}$ such that
\[
|f(\lambda v_1+\mu v_2+\eta v_k)-
f(\lambda v_1+\mu v_2)|<\vep,\;\text{whenever}\;
 |\lambda|, |\mu|, |\eta|\le1.
\]}
\end{lemma}
\begin{proof}

{Let $\{q_i\}_{i=1}^K$ be an $\frac{\varepsilon}{3L}$-dense subset of $[-1,1]$.
We show now that $N(\omega,L,\vep)=K^2M(\omega, L,\frac\vep3)+2$
suffices for the conclusion of the lemma. So, let $n$, $f$ and $\{v_j\}_{j=1}^n$ be as in the hypothesis. Then, for every
$(i,j)\in\{1,\dots,K\}^2$, the set of $k$'s such that there exists $|\eta|\le1$ for which
\[
|f(q_i v_1+q_j v_2+\eta v_k)-
f(q_i v_1+q_j v_2)|\ge\frac\vep3
\]
has cardinality at most $M(\omega,L,\frac\vep3)-1$. Indeed, apply Lemma \ref{main} to the function $\tilde f: B_{\ell_\infty^{M(\omega, L,\frac\vep3)}}\to\R$ given by 
$$
\tilde f\left(\sum_{l=1}^{M(\omega, L,\frac\vep3)}x_le_l\right)=f\left(q_iv_1+q_jv_2+\sum_{l=1}^{M(\omega, L,\frac\vep3)}x_lv_{j_l}\right),
$$
where $\{v_{k_l}\}_{l=1}^{M(\omega, L,\frac\vep3)}$ is a subsequence of $\{v_k\}_{k=3}^n$.
  So after discarding the unsuitable
sets of $k$'s for each $(q_i,q_j)$, $i,j=1,\dots, K$, we
find a suitable $k$ that fits the inequalities
\[
|f(q_i v_1+q_jv_2+\eta v_k)-
f(q_i v_1+q_jv_2)|<\frac\vep3,\;\text{whenever}\;
|\eta|\le1,
\]
for all $i,j=1,\dots, K$ simultaneously.
Finally, the conclusion is reached thanks to the density of $q_i$'s.}
\end{proof}

\begin{lemma}\label{main3}

For each modulus $\omega$ there exists
$\delta=\delta(\omega)>0$
such that if
$f\in\mathcal C^{1,+}([0,1])$ {has a} derivative $df$ {with} modulus of continuity $\omega$,
and {if} $f(x+\delta)-f(x)\ge\frac12\delta$
then ${df(t)\ge\frac14\; \text{for all}\; t\in[x,x+\delta]}$.
\end{lemma}
\begin{proof}
Since the modulus function $\omega$ is continuous and $\lim_{t\to0}\omega(t)=0$,
we can choose $\delta=\delta(\omega)>0$ so that
\[
{\omega(t)\le\frac18\;\text{for all}\;t\leq\delta.}
\]
By the fundamental theorem of calculus,
\[
f(x+\delta)=f(x)+\int_0^\delta df(x+t)dt.
\]
 {As $\omega$ is a modulus of $df$,} we have 
 \begin{equation}\label{eq:modulus}
 df(x)-\omega(t)\le df(x+t)\le df(x)+\omega(t). 
 \end{equation}
 Hence
\[
\frac12\delta\le f(x+\delta)-f(x)\le df(x)\delta+\int_0^\delta\omega(t)dt
\le df(x)\delta+\frac18\delta.
\]
{So, $df(x)\ge\frac38$ and the choice of $\delta$ along with (\ref{eq:modulus}) implies that
$df(x+t)\ge\frac14$ for all $t\in[0,\delta]$.}
\end{proof}

\section{Main results}

Our main result is the following.

\begin{theorem}\label{thm:infinite}(GCH)
Let $X$ be a Banach space, $\Gamma$ be an infinite regular cardinal.
Suppose that {$F\in\mathcal C^{1,+}(B_{\ell_\infty(\Gamma)}; X)$} is
such that $\inf_{\gamma<\Gamma}\|F(e_\gamma){-F(0)}\|>0$.
Then there exists $x\in B_{\ell_\infty(\Gamma)}$ and
$\Gamma'\subset\Gamma$ of the same cardinality as $\Gamma$
so that $dF(x)\upharpoonright_{{\ell_\infty(\Gamma')}}$
is an isomorphism. In particular, $X$ contains a copy of
$\ell_\infty(\Gamma)$.
\end{theorem}

\begin{proof}
{Passing from $F$ to $F-F(0)\chi_{B_{\ell_\infty(\Gamma)}}$, we may assume that $F(0)=0$ and that $\inf_{\gamma<\Gamma}\|F(e_\gamma)\|>0$.} 

Let $\omega$ be the modulus of continuity of $dF$.
Observe that the seminormalized long
sequence  $\{x_\alpha\}_{\alpha<\Gamma}$, where
$x_\alpha=F(e_\alpha), \alpha\in\Gamma$, is weakly null.
Indeed, the space $c_0(\Gamma)$ is a $\mathcal{W}_1$ space 
\cite[Theorem 6.30]{HJ},
hence for every $\vep>0$ and $\phi\in X^*$, the set
$\{\alpha: |\phi\circ F(e_\alpha)|>\vep\}$ is finite.

Using  Proposition \ref{w-null}, there is a bounded biorthogonal system
$\{x_\alpha, \phi_\alpha\}_{\alpha<\Gamma}$
 in $X\times X^*$, $\|x_\alpha\|\ge1, \|\phi_\alpha\|\le K$. Note that
 the family of real valued functions
 $\{\phi_\alpha\circ F(\cdot)\}_{\alpha<\Gamma}$
 is equi-uniformly differentiable with modulus $K\omega(\cdot)$ and equi-Lipschitz
 with the constant $KL$. Moreover, of course,
 $\phi_\alpha\circ F(e_\beta)=\delta_{\alpha,\beta}$.

 Fix $\delta=\delta(K\omega)$ from Lemma \ref{main3}, $\vep=\frac{\delta}{20K}$,
  {and $N>N(K\omega,KL,\vep,)$, where $N(K\omega,KL,\vep,)$ is from Lemma \ref{main2}.}

As a first step we replace the index set $\Gamma$
by a partially ordered set $(T,\preceq)$ of the same cardinality.
For $\alpha<\Gamma$ we denote $T_{\alpha}=\{t:[0,\alpha]\to\{1,\dots,N\}\}$.
We let $T=\cup_{\alpha<\Gamma}T_\alpha$.
Since $\card(T_\alpha)=\{1,\dots,N\}^{\alpha}=\alpha^+\le \Gamma$, using the GCH,
it follows from our assumptions that
$\card(T)=\card(\cup_{\alpha<\Gamma} T_\alpha)=\Gamma$ (see 
\cite[p. 48]{J} for
details on these standard facts). Note that if
$\Gamma=\omega_{{0}}$ then all $T_\alpha$ are in fact finite sets and
the assumption GCH is not needed.

The partial order is defined as follows. Let $s\in T_\beta$, $t\in T_\alpha$.
Then $s\preceq t$ if and only if
$\beta\le\alpha$ and $ t\upharpoonright_{[0,\beta]}=s$. We will keep
the notation $e_t, t\in T$, for the unit vectors in $\ell_\infty(T)$.
We will also denote by $\phi_{t}$ the corresponding biorthogonal
functionals originally indexed by indices  $\alpha<\Gamma$.
For convenience, we define
$T^t_\alpha=\{s\in T_\alpha: s\succeq t\}$, whenever
$t\in T_\beta, \beta<\alpha$, and we also let
$T^t=\cup_{\alpha>\beta} T^t_\alpha$.
Clearly, $T^t\ne\emptyset$ for any $t\in T$.

It is clear that $T^s_\alpha$ is disjoint from $T^t_\alpha$, unless
$s\preceq t$ (resp. $t\preceq s$),
in which case $T^t_\alpha\subset T^s_\alpha$
(resp. $T^s_\alpha\subset T^t_\alpha$). Similarly,
$T^t$ is a subset of $T^s$ {(resp. $T^s\subset T^t$)} if and only if $s\preceq t$ {(resp. $t\preceq s$)}, otherwise they
are disjoint.

We begin a transfinite  construction in $\alpha<\Gamma$
of transfinite sequences $\{t_\alpha\}_{\alpha<\Gamma}$,
$\{m_\alpha\}_{\alpha<\Gamma}$, and
 $\{z^j_\alpha\}_{\alpha<\Gamma, j=1,\dots {N}}$ so that
 the following conditions are satisfied.
\vvs

1. $t_\alpha\in T_\alpha$ is an increasing long
sequence, i.e. $t_\alpha\succeq t_\beta$ for $\alpha\ge\beta$.

2. $m_\alpha\in\{1,\dots, N\}$.

3.

\begin{equation*}
 z^j_\alpha=\begin{cases} e_{t_\alpha}& \text{ for } j=m_\alpha\\
\\0&\text{ for } j\ne m_\alpha.\endcases
\endgroup
\end{equation*}

4. If $|\lambda|,|\eta|\le1$, $z^j_{\alpha}\ne0$, $\|u\|\le1$, and
$\text{supp}(u)\subset {\{t_{\alpha+1}\}\cup} T^{t_{\alpha+1}}$, then

\[
|\phi_{t_\alpha}\circ F(\eta\sum_{\beta<\alpha} z^j_\beta+\lambda z^j_\alpha+u)-
\phi_{t_\alpha}\circ F(\eta\sum_{\beta<\alpha} z^{{j}}_\beta+\lambda z^j_{\alpha})|\le\vep.
\]

5. If $|\lambda|,|\eta|\le1$, $z^j_{\alpha}\ne0$,  then

\[
|\phi_{t_\alpha}\circ F(\eta\sum_{\beta<\alpha} z^j_\beta+\lambda z^j_\alpha)-
\phi_{t_\alpha}\circ F(\lambda z^j_{\alpha})|\le\vep.
\]

First step for $\alpha=0$. Choose an arbitrary $t_0\in T_0$, $m_0=1$,
$z^1_0=e_{t_0}$, $z^j_0=0$
for $j=2,\dots, N$. This choice satisfies the conditions 1{--}3 and 5.
In order to satisfy the condition 4 for $\alpha=0$ we will need
to select the appropriate $t_1$. We proceed by contradiction, using
 Lemma \ref{main2}. Indeed, the set of available extensions
 $t_1\succeq t_0$ has $N$ elements $s_1,\dots, s_N$. Their corresponding
 sets $\{s_j\}\cup T^{s_j}, j\in\{1,\dots,N\}${,} are pairwise disjoint.
 Assuming that for
 every $j\in\{1,\dots,N\}$ there exists {a} $u_j\in\ell_\infty(T)$ with {the} support
 contained in $\{s_j\}\cup T^{s_j}$ and $\|u_j\|\le1$ which violates
 the condition 4,
  we reach a contradiction
 with Lemma \ref{main2} in the setting {$v_1=0$, $v_2=e_{t_0}$,
 $v_j=\frac1{\|u_j\|}u_j$, $j=3,\dots,N$, and $f=\phi_{t_0}\circ F$.}
   So there exists an extension $t_1$
 such that 4 holds.

 Inductive step to reach $\alpha=\Lambda<\Gamma$.
 Suppose we have already constructed the system
 $\{t_\alpha\}_{\alpha<\Lambda}$,
$\{m_\alpha\}_{\alpha<\Lambda}$, and
 $\{z^j_\alpha\}_{\alpha<\Lambda, j=1,\dots N}$
 and all conditions 1{--}3 and 5 are satisfied for $\alpha<\Lambda$, and
 condition 4 is satisfied if $\alpha+1<\Lambda$.
 First we wind $t_\Lambda$ so that 4 will hold for all
 $\alpha+1<\Lambda+1$.

In case $\Lambda$ is a limit ordinal,
denote by $r:[0,\Lambda)\to\{1,\dots,N\}$ the function
such that $r\upharpoonright_{[0,\alpha]}=t_\alpha, \alpha<\Lambda$.
This function is well-defined, as the sequence $\{t_\alpha\}_{\alpha<\Lambda}$
is $\preceq$-increasing. We let

\begin{equation*}
t_{\Lambda}(\alpha)=\begin{cases} r(\alpha)& \text{ for } \alpha<\Lambda\\
\\1&\text{ for } \alpha=\Lambda{.}\endcases
\endgroup
\end{equation*}

Observe that $T_\Lambda\subset T_\alpha$ for all $\alpha<\Lambda$ and
so the condition 4 will remain valid for all $\alpha+1<\Lambda+1$. 

In the non-limit case $\Lambda=\gamma+1$ we denote again
by $s_1,\dots, s_N$ the set of all immediate followers of $t_\gamma$.
We claim that there is some $t_\Lambda=s_j$ so that for every $u$ of norm not
exceeding $1$ and supported
by $\{t_\Lambda\}\cup T^{t_\Lambda}$  we have

\[
|\phi_{t_\gamma}\circ F(\eta\sum_{\beta<\gamma}z^{m_\gamma}_\beta +\lambda
 z^{m_\gamma}_\gamma+
u)-\phi_{t_\gamma}\circ F(\eta\sum_{\beta<\gamma}z^{m_\gamma}_\beta +\lambda
 z^{m_\gamma}_\gamma)
|\le\vep.
\]

Arguing by contradiction again, it suffices to apply
 Lemma {\ref{main2}} in the setting $f=\phi_{t_\gamma}\circ F$,
{$v_1=\sum_{\beta<\gamma} z^{m_\gamma}_\beta$, $v_2=z^{m_\gamma}_\gamma$ and $v_j, j=3,\dots, N,$} 
be any
vector of norm at most one supported by $\{s_j\}\cup T^{s_j}$.

Next, in order to find $m_\Lambda$,
apply Lemma \ref{main2} in the setting
{$v_1=0$, $v_2=e_{t_\Lambda}$, $v_j=\sum_{\beta<\Lambda} z^j_\beta$,
$j\in\{3,\dots N\}$,
$f=\phi_{t_\Lambda}\circ F$.
We obtain that there exists some $m_\Lambda\in\{3,\dots,N\}$} such that
if $|\lambda|,|\eta|\le1$ then
\[
|\phi_{t_\Lambda}\circ F(\eta\sum_{\beta<\Lambda} z^{m_\Lambda}_\beta+
\lambda e_{t_\Lambda})-
\phi_{t_\Lambda}\circ F(\lambda e_{t_\Lambda})|\le\vep.
\]

To finish the inductive step, let

\begin{equation*}
z_{\Lambda}^{j}=\begin{cases} e_{t_\Lambda}& \text{ for }
j=m_\Lambda\\
\\0&\text{ for }j\ne m_\Lambda.\endcases
\endgroup
\end{equation*}

It is easy to verify that the inductive assumption has been recovered,
i.e. shifted by one.
This completes the  inductive step.

If $|\lambda|,|\eta|\le1$,
$\text{supp}(u)\subset {\{t_{\alpha+1}\}\cup T^{t_{\alpha+1}}}$, $\|u\|\le1$, then

\[
|\phi_{t_\alpha}\circ F(\sum_{\beta<\alpha}\eta z^{m_\alpha}_\beta+\lambda z^{m_\alpha}_\alpha+u)-
\phi_{t_\alpha}\circ F(\lambda z^{m_\alpha}_\alpha)|\le
\]
\[
|\phi_{t_\alpha}\circ F(\sum_{\beta<\alpha}\eta z^{m_\alpha}_\beta+\lambda z^{m_\alpha}_\alpha+u)-
\phi_{t_\alpha}\circ F(\sum_{\beta<\alpha}\eta z^{m_\alpha}_\beta+\lambda z^{m_\alpha}_\alpha)|+
\]
\[
|\phi_{t_\alpha}\circ F(\sum_{\beta<\alpha}\eta z^{m_\alpha}_\beta+\lambda z^{m_\alpha}_\alpha)-
\phi_{t_\alpha}\circ F(\lambda z^{m_\alpha}_\alpha)|\le2\vep.
\]

Now choose $j\in\{1,\dots,N\}$ such that $\card\{\alpha: z^j_\alpha\ne0\}=\card\Gamma$.
Of course, the set $\{\alpha: z^j_\alpha\ne0\}$ is order isomorhic to $\Gamma$,
so we may reindex it to become again $\Gamma$, and obtain as a result
a long sequence $\{z_\alpha\}_{\alpha<\Gamma}$ of unit vectors in
$\ell_\infty(\Gamma)$ such that the following condition is satisfied.
If $|\lambda|,|\eta|\le1$,
$\text{supp}(u)\subset[\alpha+1,\Gamma)$, $\|u\|\le1$, then

\[
|\phi_\alpha\circ F(\sum_{\beta<\alpha}\eta z_\beta+\lambda z_\alpha+u)-
\phi_\alpha\circ F(\lambda z_\alpha)|\le2\vep.
\]

Recall that $\phi_\alpha\circ F(z_\alpha)-\phi_\alpha\circ F(0)\ge1$,
$\alpha<\Gamma$. By our choice of
{$\delta=\delta(K\omega)$ so that $K\omega(t)<\frac18$ for all $t\leq\delta$,} there exists
an interval $[p_\alpha, q_\alpha]\subset[0,1]$,
$q_\alpha=p_\alpha+\delta,$ such that
$\phi_\alpha\circ F(q_\alpha z_\alpha)-\phi_\alpha\circ
F(p_\alpha z_\alpha)\ge\frac34\delta$.
Hence, if $|\eta|\le1$,
$\text{supp}(u)\subset[\alpha+1,\Gamma)$, and $\|u\|\le1$, then

\[
|\phi_\alpha\circ F(\sum_{\beta<\alpha}\eta z_\beta+q_\alpha z_\alpha+u)-
\phi_\alpha\circ F(\sum_{\beta<\alpha}\eta z_\beta+ p_\alpha z_\alpha+u)|\ge
\]
\[
|\phi_\alpha\circ F(q_\alpha z_\alpha)-
\phi_\alpha\circ F(p_\alpha z_\alpha)|-
|\phi_\alpha\circ F(\sum_{\beta<\alpha}\eta z_\beta+q_\alpha z_\alpha+u)-
\phi_\alpha\circ F(q_\alpha z_\alpha)|-
\]
\[
-|\phi_\alpha\circ F(\sum_{\beta<\alpha}\eta z_\beta+p_\alpha z_\alpha+u)-
\phi_\alpha\circ F(p_\alpha z_\alpha)|
\ge\frac34\delta-4\vep\ge\frac12\delta
\]

Clearly, there exists $\lambda\in[0,1]$ such that the set
$\Lambda=\{\gamma\in\Gamma: \lambda\in[p_\gamma,q_\gamma]\}$ has the same
cardinality as $\Gamma$.
By Lemma \ref{main3} we conclude that for every $\alpha\in\Lambda$

\[
\|dF(\sum_{\beta<\Gamma}\lambda z_\beta)[z_\alpha]\|\ge
\frac1K d(\phi_\alpha\circ F)(\sum_{\beta<\Gamma}\lambda z_\beta)[z_\alpha]
\ge{\frac1{4K}.}
\]

To finish the proof, it suffices to apply the original Theorem \ref{rose}
of Rosenthal for linear operators.

\end{proof}

{\bf Remark.}
It should be noted that the proof of the main theorem
for $\Gamma=\omega_0$ is valid in ZFC. Indeed, the proof only
uses the fact (in general depending on the GCH)
that $\Gamma\ge2^\alpha$ for every cardinal $\alpha<\Gamma$.
This condition is clearly satisfied for $\Gamma=\omega_0$.
Also, by inspection of the proof,
 it should be noted that $T=dF(x)\upharpoonright_Y$
is an isomorphism such that $l\|y\|\le\|T(y)\|\le L\|y\|, y\in Y$,
where $l=l(\omega,\vep,L)>0$ is independent of
the concrete form of $F$.

As a corollary, it follows that there does not exist any uniformly
differentiable
mapping from $\ell_\infty$ into $c_0$ which fixes the basis.
This generalises the classical theorem of Phillips which {says} that $c_0$
is not complemented in $\ell_\infty$ (see \cite[Theorem 5.6]{FHHMZ}).
~\\

\begin{corollary}(GCH)
Let $X$ be a Banach space, $\Gamma$ be a singular cardinal.
Suppose that {$F\in\mathcal C^{1,+}(B_{\ell_\infty(\Gamma)}; X)$ is such that $\inf_{\gamma<\Gamma}\|F(e_\gamma)-F(0)\|>0$.}
 Then for every cardinal $\Lambda<\Gamma$ there exists
$x\in B_{\ell_\infty(\Gamma)}$ and a
subspace $Y\subset\ell_\infty(\Gamma)$ isomorphic to $\ell_\infty(\Lambda)$
such that $dF(x)\upharpoonright_Y$ is an ismorphism.
\end{corollary}

{We now recall the notion of ultraproduct following Section 4.1 of \cite{HJ}.
Let $X_n$ be Banach spaces and let $\ell_\infty(\en;\{X_n\})$ be the Banach space
$$
\left\{(x_n)_{n=1}^\infty, x_n\in X_n, \sup\{\|x_n\|,n\in\en\}<\infty\right\}
$$
with the norm given by $\sup\{\|x_n\|,n\in\en\}$. If $\mathcal U$ is an
ultrafilter on $\en$, we define the \emph{ultraproduct of $\{X_n\}$} as the
quotient space
$$
(X_n)_{\mathcal U}=\left.\ell_\infty(\en;\{X_n\})\middle/
\left\{(x_n)_{n=1}^\infty\in\ell_\infty(\en;\{X_n\}),
\lim_{\mathcal U}\|x_n\|=0\right\}\right.
$$
endowed with the canonical quotient norm. Here
$\lim_{\mathcal U}\|x_n\|\in\er$ is the limit with respect to the
ultrafilter $\mathcal U$. Then, $(X_n)_{\mathcal U}$ is a Banach space and
$\|(x_n)_{\mathcal U}\|=\lim_{\mathcal U}\|x_n\|$ for every
$(x_n)_{\mathcal U}\in (X_n)_{\mathcal U}$ represented by
$(x_n)_{n=1}^\infty\in\ell_\infty(\en;\{X_n\})$.
Here comes the finite-dimensional result.}

\begin{theorem}\label{thm:finite}
For each modulus $\omega$, $L>0$, $m\in\en$
and $\vep>0$, there is $l=l(\omega,L,\vep)>0$ and
$N(\omega,L,m,\vep)\in\en$ such that
if $n\geq N(\omega,L,m,\vep)$, $Y$ is a Banach space and
\mbox{$f\in\mc C^{1,+}(B_{\ell_\infty^n};Y)$} is an
$L$-Lipschitz mapping whose derivative $df$ has modulus of
continuity~$\omega$
and for which $\norm{f(e_i)-f(e_j)}\geq\vep$ for all
$i,j\in\left\{1,\dots,n\right\}$, $i\neq j$,
then there exists \mbox{$\mathcal J\subset \left\{1,\dots,n\right\}$}
with $\card(\mathcal J)=m$ and $x\in\ell_\infty^n$ such that
$df(x)|_{\ell_\infty(\mathcal{J})}$ is an isomorphism
with $l\|y\|\le\|T(y)\|\le L\|y\|$.
\end{theorem}
\begin{proof}
{Let  $\omega$ be a modulus and} let $L>0$, $m\in\en$
and $\vep>0$. Suppose that $(f_n)_{n=1}^\infty$ is a sequence of
mappings such that for every $n\in\en$, $f_n\in\mc C^{1,+}
(B_{\ell_\infty^n};Y_n)$ for some Banach space
$Y_n$, $f_n$ is $L$-Lipschitz, $f_n(0)=0$, $df_n$ has modulus of
continuity~$\om$, and $\norm{f_n(e_i)-f_n(e_j)}\ge\vep$ for each
$i,j\in\left\{1,\dots,n\right\}$, $i\neq j$. We show that then there
exists \mbox{$n_0\in\en$,} a set
$\mathcal J\subset\left\{1,\dots,n_0\right\}$ of cardinality $m$ and a
point $x\in\ell_\infty^{n_0}$ such that
$df_{n_0}(x)|_{\ell_\infty(\mathcal{J})}$ is an isomorphism.
Since $(f_n)_{n=1}^\infty$ is an arbitrary sequence with the listed
properties, the statement of the theorem follows.

{We may regard all $f_n$'s as mappings from $B_{\ell_\infty}$ by composing with the projections
$P_n\colon\ell_\infty\to\ell_\infty^n$ given by
$P_n\left((x_i)_{i=1}^\infty\right)=(x_i)_{i=1}^n$.} Let $\mc U$ be a free ultrafilter
on $\en$. We can define a mapping
$f\colon B_{\ell_\infty}\to (Y_n)_{\mc U}$ by
$$
f\left(x\right)=\left(f_n(x)\right)_{\mc U}
$$
for $x\in B_{\ell_\infty}$. Since the mappings $f_n$ are equi-Lipschitz,
$f$ is well-defined.

Moreover, $f\in\mc C^{1,+}(B_{\ell_\infty};(Y_n)_{\mc U})$.
Indeed, Corollary 1.99 in \cite{HJ} says in particular that
$$
\left\|g(x+u)-g(x)-dg(x)[u]\right\|\leq\sup_{t\in[0,1]}\|dg(x+tu)-dg(x)\|
\|u\|
$$
for every $g\in\mathcal C^1(U;Y)$, where $U$ is an open convex subset of a
Banach space $X$ and $Y$ is a Banach space, and for every $x\in U$ and every
$u\in X$ such that $x+u\in U$.
Applying {this} to the functions $f_n$, we obtain
\begin{equation}\label{eq:taylor for n}
\normb{big}{f_n(x+u)-f_n(x)-df_n(x)[u]}\le\om(\norm u)\norm u
\end{equation}
for every $x\in \Int B_{\ell_\infty}$ and $u\in \ell_\infty$ such that
$x+u\in \Int B_{\ell_\infty}$, and every $n\in\en$.
Take $x\in \Int B_{\ell_\infty}$. Define $S\colon \ell_\infty\to
(Y_n)_{\mc U}$ by $S(u)=\bigl(df_n(x)[u]\bigr)_{\mc U}$ for
$u\in\ell_\infty$.
It is easy to see that $S$ is a bounded linear operator from $\ell_\infty$
into $(Y_n)_{\mathcal U}$. Let $u\in \ell_\infty$ be such that $x+u\in \Int B_{\ell_\infty}$.
Then by (\ref{eq:taylor for n}),
\begin{align*}
\norm{f(x+u)-f(x)-S(u)}&=\normb{big}{\bigl(f_n(x+u)-f_n(x)-
df_n(x)[u]\bigr)_{\mc U}}\\
&\leq\om(\norm{u})\norm{u}.
\end{align*}
From Theorem 1.114 in \cite{HJ} it follows that
$f\in\mathcal C^{1,+}(B_{\ell_\infty};(Y_n)_{\mc U})$ and that the
modulus of continuity of $df$ is $\tau\omega$ for some constant
$\tau\geq1$.

Besides, $\left\{f(e_k), k\in\en\right\}$ is not relatively compact in
$(Y_n)_{\mc U}$ as
$$
\|f(e_k)-f(e_l)\|=\|(f_n(e_k)-f_n(e_l))_{\mc U}\|\geq\vep
$$
for all $k,l\in\en$, $k\neq l$.
So, we can apply Theorem \ref{thm:infinite} to the mapping $f$. We
obtain an infinite set $\mathcal K$ of natural numbers and a point
$x\in \text{Int}{B_{\ell_\infty(\mathcal{K})}}$ such that 
$df(x)|_{\ell_\infty(\mathcal{K})}$ is an isomorphism.  From the proof of Theorem
\ref{thm:infinite} it follows that there exists $\xi=\xi(\omega,L,\varepsilon)$
such that $\|df(x)[z]\|\geq\xi\|z\|$ for every $z\in {\ell_\infty(\mathcal{K})}$. Denote
$\mathcal J\subset\mathcal K$ the set of the first $m$ elements of
$\mathcal K$.

If $\zeta>0$ satisfies that $\omega(\zeta)<\frac{\xi}{4\tau}$ and that
$x+\zeta B_{\ell_\infty(\mathcal{K})}\subset \Int B_{\ell_\infty(\mathcal{K})}$ and if $u\in B_{\ell_\infty(\mathcal{K})}$, $\|u\|=\zeta$,
then by Corollary 1.99 in \cite{HJ},
\begin{align*}
\|f(x+u)-f(x)\|\geq\|df(x)[u]\|-
\tau\omega(\|u\|)\|u\|\geq(\xi-\tau\omega(\zeta))\zeta>\frac{3}{4}\xi
\zeta.
\end{align*}
Choose $q\in\en$ so that $\frac 1q<\frac{\xi}{4L}$. Denote
$$
Q=\left\{\sum_{j\in\mathcal J}\sigma_j\beta_je_j, \sigma_j\in\{-1,1\}
\textup{ and }\beta_j\in\left\{0,\frac{1}{q}\zeta,\dots,\frac{q-1}{q}\zeta,
\zeta\right\}\textup{~for all~} j\in\mathcal J\right\}.
$$
For each $v\in Q$, $\|v\|=\zeta$, the set
$$N_v=\left\{n\in\en,\|f_n(x+v)-f_n(x)\|>\frac{3}{4}\xi\zeta\right\}$$
belongs to the ultrafilter $\mathcal U$. Therefore the intersection
$\bigcap_ {v\in Q,\|v\|=\zeta}N_v$ is an infinite set. Take
$n_0\in\bigcap_ {v\in Q, \|v\|=\zeta}N_v$ such that $n_0\geq\max \mc J$.
Then, given $u\in B_{\ell_\infty(\mathcal{J})}$, $\|u\|=\zeta$, we find
$v\in Q$, $\|v\|=\zeta$, so that $\|u-v\|\leq\frac 1q\zeta$ and obtain that
\begin{align*}
\left\|f_{n_0}(x+u)-f_{n_0}(x)\right\|&\geq\left\|f_{n_0}\left(x+v\right)-
f_{n_0}(x)\right\|-\left\|f_{n_0}(x+u)-f_{n_0}\left(x+v\right)\right\|\\
&>\frac{3}{4}\xi\zeta-L\frac 1q\zeta>\frac{1}{2}\xi\zeta.
\end{align*}
In view of the Corollary 1.99 in \cite{HJ} again, for
$u\in B_{\ell_\infty(\mathcal{J})}$, $\|u\|=\zeta$, we have that
\begin{align*}
\|df_{n_0}(x)[u]\|\geq \left\|f_{n_0}(x+u)-f_{n_0}(x)\right\|-
\omega(\zeta)\zeta>\frac{1}{2}\xi\zeta-\frac{1}{4\tau}\xi\zeta\geq
\frac{1}{4}\xi\zeta.
\end{align*}
Hence, $df_{n_0}(x)|_{\ell_\infty(\mathcal{J})}$ is an
isomorphism with a lower bound depending only
on the modulus $\omega$, the Lipschitz constant $L$ of $f$, and $\vep$. This finishes the proof.
\end{proof}

We conclude by deriving a corollary that witnesses the relation between
just proved Theorem \ref{thm:finite} and Lemma 6.27 in \cite{HJ}.
\begin{corollary}
\label{cor:vector lemma}
Let $\omega$ be a modulus, $L>0$, $\vep>0$ and
let $Y$ be a Banach space with non-trivial cotype. Then there is an
$N(\omega,L,\vep,Y)\in\en$ such that if $n\geq N(\omega,L,\vep,Y)$ and if
$f\in\mathcal C^{1,+}(B_{\ell_\infty^n};Y)$ is an $L-$Lipschitz mapping
whose derivative $df$ has modulus of continuity $\omega$, then there
exist $i,j\in\{1,\dots,n\}$, $i\neq j$, for which $\|f(e_i)-f(e_j)\|<\vep$.
\end{corollary}
\begin{proof}
Suppose that the statement does not hold. Then for every $m\in\en$ there
is
{$n\geq N(\omega,L,m,\vep)$,} where $N(\omega,L,m,\vep)$ is the constant
obtained in Theorem \ref{thm:finite}, and there is $f\in\mathcal
C^{1,+}(B_{\ell_\infty^n};Y)$ which is $L-$Lipschitz and such that its
derivative $df$ has modulus of continuity $\omega$ and that
$\|f(e_i)-f(e_j)\|\geq\vep$ for all $i,j\in\{1,\dots,n\}$, $i\neq j$.
So, according to Theorem \ref{thm:finite}, $Y$ contains a subspace $Y_m$
isomorphic to $\ell_\infty^m$. Moreover, from the proof of
Theorem \ref{thm:finite} it follows that
$d(Y_m,\ell_\infty^m)\leq C(\omega,L,\vep)$, where
$d(Y_m,\ell_\infty^m)$ is the Banach-Mazur distance of
$Y_m$ and $\ell_\infty^m$ and $C(\omega,L,\vep)$ is a constant
which depends on $\omega,\,L,\,\vep,$ but does not depend on $m$.
This is by Maurey and Pisier's characterization \cite{MP} a
contradiction with the non-trivial cotype of $Y$.
\end{proof}

\end{document}